\documentclass[11pt]{article}

\usepackage[letterpaper,margin=1in]{geometry}
\usepackage{amsmath,amssymb,amsthm,mathtools}
\usepackage{booktabs}
\usepackage{graphicx}
\usepackage{microtype}
\usepackage{xcolor}
\usepackage{fontspec}
\usepackage{float}
\DeclareFontFamily{U}{msb}{}
\DeclareFontShape{U}{msb}{m}{n}{<->msbm10}{}
\usepackage[hidelinks]{hyperref}
\hypersetup{
  pdftitle={A Superdiffusive Local Limit Theorem for the Elephant Random Walk and the Breakdown of Log-Concavity},
  pdfauthor={Hélio Trinas and Glauco Valle}
}

\numberwithin{equation}{section}

\newtheorem{theorem}{Theorem}[section]
\newtheorem{proposition}[theorem]{Proposition}
\newtheorem{lemma}[theorem]{Lemma}

\theoremstyle{remark}
\newtheorem{remark}[theorem]{Remark}

\newcommand{\E}{\mathbb{E}}
\newcommand{\Pp}{\mathbb{P}}
\newcommand{\R}{\mathbb{R}}
\newcommand{\Z}{\mathbb{Z}}

\title{A Superdiffusive Local Limit Theorem for the Elephant Random Walk\\
and the Breakdown of Log-Concavity}
\author{Hélio Trinas\thanks{E-mail: \href{mailto:helio@dme.ufrj.br}{helio@dme.ufrj.br}. Hélio Trinas was supported by CAPES.} \and Glauco Valle\thanks{E-mail: \href{mailto:glauco.valle@im.ufrj.br}{glauco.valle@im.ufrj.br}. Glauco Valle was supported by CNPq grants 307938/2022-0 and 403423/2023-6, FAPERJ grant E-26/200.442/2023.}}
\date{}

\newcommand{\keywordsandmsc}{%
  \begingroup
  \renewcommand{\thefootnote}{}%
  \footnotetext{%
    \noindent\textbf{Keywords.}
    Elephant random walk; local limit theorem; superdiffusive regime;
    log-concavity; non-Gaussian limit law.

    \noindent\textbf{2020 Mathematics Subject Classification.}
    60K37.%
  }%
  \endgroup}

\begin{document}
\maketitle
\keywordsandmsc

\begin{abstract}
Let $(S_n)_{n\ge 1}$ be the one-dimensional Elephant Random Walk (ERW) in the superdiffusive regime, i.e. with memory parameter $p\in(3/4,1)$, and first-step bias $\Pp(S_1=1)=q$. Set
$a=2p-1\in(1/2,1)$. If $f_{q,a}$ denotes the density of the superdiffusive limit
$\mathcal{L}_{q,a}$, we prove the uniform local limit theorem
\[
\lim_{n\to \infty} \sup_{\substack{j\in\Z\\ j\equiv n\ ({\rm mod}\ 2)}}
 \left|\frac{n^a}{2}\Pp(S_n=j)-f_{q,a}\!\left(\frac{j}{n^a}\right)\right| = 0.
\]
The proof relies on a feature specific to the classical ERW: the exact
recurrence relation for the probability mass function (p.m.f.) of $S_n$. This
recurrence yields a uniform $O(n^{-a})$ bound for the p.m.f. and a uniform $O(n^{-2a})$ bound for its first discrete
differences. A piecewise-linear interpolation argument then upgrades
the known weak convergence to uniform local convergence.

We next disprove a conjecture of Gu\'erin, Laulin, Raschel and Simon
\cite[Remark 2.3]{guerin-laulin-raschel-simon} concerning the walk with
$q=1$. Writing $\mathcal{L}_a:=\mathcal{L}_{1,a}$ and
$f_a:=f_{1,a}$, define
\[
a_{\rm ev}
:=
\sup\left\{
b\in(1/2,1]:
\begin{array}{c}
\text{the p.m.f. of $S_n$ is eventually log-concave}\\
\text{for every $a\in(1/2,b)$}
\end{array}
\right\},
\]
and
\[
a_\star
:=
\inf\left\{
b\in(1/2,1):
f_a\text{ is not log-concave for every }a\in(b,1)
\right\}.
\]
The conjecture in \cite[Remark 2.3]{guerin-laulin-raschel-simon}
is equivalent to $a_{\rm ev}=1$. Their results imply that
$a_{\rm ev}\geq(\sqrt{5}-1)/2$, and we prove
\[
\frac{\sqrt{5}-1}{2}
\leq a_{\rm ev}
\leq \min\{0.80399,a_\star\} \le a_\star < 0.918.
\]
The bounds $a_\star<0.918$ and $a_{\rm ev}\leq0.80399$ follow,
respectively, from the first three exact moments of $\mathcal{L}_a$
together with the sharp skewness inequality for centered log-concave
distributions, and from a certified finite-order analysis of the exact
p.m.f. recurrence. Finally, direct numerical iterations of the
recurrence provide evidence that $a_{\rm ev}\approx0.80399$,
equivalently $p_{\rm ev}\approx0.902$.
\end{abstract}

\section{Introduction}

The Elephant Random Walk (ERW) was introduced by Sch\"utz and Trimper
\cite{schutz-trimper}. It is a non-Markovian nearest-neighbor walk whose next
increment is obtained by selecting one of the previous increments uniformly and
then either repeating or reversing it. In one dimension the memory parameter
$p=3/4$ marks the transition between the diffusive and superdiffusive regimes.
When $p>3/4$ and $a=2p-1$, the normalized position $S_n/n^a$ converges almost
surely to a non-Gaussian random variable $\mathcal{L}_{q,a}$; see
\cite{baur-bertoin,bercu}. Kubota and Takei \cite{kubota-takei} obtained Gaussian
fluctuations around the random leading term.

The distribution of $\mathcal{L}_{q,a}$ is now understood at a finer level. Gu\'erin,
Laulin and Raschel \cite[Theorems 1.3 and 1.4]{guerin-laulin-raschel} proved that
it has a positive bounded smooth density on the whole line and derived a nonlinear
recurrence for its moments. Conditioning on the first increment also gives
$\mathcal{L}_{q,a}\stackrel{d}{=}(2\xi_q-1)\mathcal{L}_a$, where $\mathcal{L}_a:=\mathcal{L}_{1,a}$ and
$\xi_q\sim\operatorname{Bernoulli}(q)$ is independent of $\mathcal{L}_a$. Gu\'erin, Laulin,
Raschel and Simon \cite[Theorem 1.1, Proposition 2.2 and Corollary 2.4]
{guerin-laulin-raschel-simon} established unimodality and log-concavity results for the density of $\mathcal{L}_a$.

Weak convergence of $S_n/n^a$, however, does not determine the asymptotics for the probability mass function (p.m.f.) of $S_n$. Fan, Hu and Ma \cite{fan-hu-ma} obtained local limit
estimates in the diffusive and critical regimes $p\leq3/4$. More recently,
Peres and Qin \cite{peres-qin} proved general upper bounds for transition
probabilities of step-reinforced walks, a framework that includes the classical
ERW; these are bounds rather than a local asymptotic at the
natural superdiffusive scale. To the best of our knowledge, no previous uniform
local limit theorem was available for the classical ERW in the genuinely
superdiffusive regime. Our first result fills this gap. Its proof is driven by the exact recurrence for the
p.m.f. of $S_n$. The same recurrence closes for its first discrete difference,
giving exactly the regularity needed for interpolation. No higher-order difference
estimate is used.

Our local limit theorem also sheds light on a conjecture concerning the
shape of the finite-time distributions, formulated in
\cite[Remark 2.3]{guerin-laulin-raschel-simon}. Consider the walk with
$q=1$. For every integer $r\geq0$, the authors proved that there exists
a threshold $a_r$ such that the p.m.f. of $S_n$ is log-concave for all
$n\geq r+3$ if and only if $a\leq a_r$. The sequence $(a_r)_{r\geq0}$
is increasing, and they conjectured that $a_r\uparrow1$. Thus, their
conjecture would imply that, for every fixed $a\in(1/2,1)$, the
finite-time distributions are eventually log-concave.

Since $(a_r)_{r\geq0}$ is increasing and bounded above by $1$, define
\[
a_{\mathrm{ev}}
:=
\lim_{r\to\infty}a_r
=
\sup\left\{
b\in(1/2,1]:
\begin{array}{c}
\text{the p.m.f. of $S_n$ is eventually log-concave}\\
\text{for every $a\in(1/2,b)$}
\end{array}
\right\}.
\]
The conjecture is therefore precisely the assertion that $a_{\mathrm{ev}}=1$. To describe the loss of log-concavity of the limiting distribution, write $f_a:=f_{1,a}$ and define
\[
a_\star
:=
\inf\left\{
b\in(1/2,1):
f_a\text{ is not log-concave for every }a\in(b,1)
\right\}.
\]
Our local limit theorem implies that log-concavity of the p.m.f. along
an unbounded sequence of times forces $f_a$ to be log-concave.
Consequently, $a_{\mathrm{ev}}\leq a_\star$.

Using the first three exact moments of $\mathcal{L}_a$ and the sharp
skewness inequality for centered log-concave distributions, we show
that $f_a$ is not log-concave for every $a$ above the relevant solution
of the equality case in the skewness bound. This solution is
approximately $0.917665$, and hence $a_\star<0.918$. Separately, a certified finite-order analysis of the exact p.m.f. recurrence gives
$a_{\mathrm{ev}}\leq0.80399$. Combining these estimates with the lower bound from \cite[Proposition 2.2]{guerin-laulin-raschel-simon}, we obtain
\[
\frac{\sqrt5-1}{2}
\leq a_{\mathrm{ev}}
\leq\min\{0.80399,a_\star\} \le 
a_\star<0.918.
\]
In particular, both $a_{\mathrm{ev}}<1$ and $a_\star<1$, and the
conjecture is disproved.

These thresholds leave three possible regimes. If
$a<a_{\mathrm{ev}}$, the p.m.f. is eventually log-concave and $f_a$ is
log-concave. If $a_{\mathrm{ev}}<a\leq a_\star$, the p.m.f. is not
eventually log-concave, while the log-concavity of $f_a$ is not
determined by the present arguments. If $a>a_\star$, then $f_a$ is not
log-concave and the transfer principle implies that only finitely many
finite-time rows can be log-concave. The intermediate regime may be
empty if $a_{\mathrm{ev}}=a_\star$.

Finally, direct numerical iterations of the exact p.m.f. recurrence
provide evidence that $a_{\mathrm{ev}}\approx0.80399$, $p_{\mathrm{ev}}=(1+a_{\mathrm{ev}})/2\approx0.902$. The exact values of $a_{\mathrm{ev}}$ and $a_\star$, as well as the behaviour at the boundary points, remain open.

\section{The ERW and the p.m.f. recurrence}

Let $(X_n)_{n\geq1}$ be a sequence of increments taking values in $\{-1,+1\}$, with $\Pp(X_1=1)=q\in[0,1]$, 
and set $S_0=0$ and $S_n=X_1+\cdots+X_n$ for $n\geq1$. At time $n+1$, an index is chosen uniformly at random from $\{1,\ldots,n\}$. The corresponding increment is repeated with probability $p\in(0,1)$ and its sign is reversed with probability $1-p$. We parametrize the model by $a=2p-1\in(-1,1)$ rather than by $p$. Thus, $(S_n)_{n\geq0}$ is the ERW with memory parameter $a$ and initial bias $q$. Notice that $a>0$ favors repetition, whereas $a<0$ favors sign reversal. The ERW is in the superdiffusive regime when $a\in(1/2,1)$.


Define
\[
u_{n,k}:=\Pp(S_n=2k-n),\qquad k\in\Z,
\]
and note that $u_{n,k}=0$ whenever $k\notin\{0,\ldots,n\}$.
On the event $\{S_n=2k-n\}$, exactly $k$ of the first $n$
increments are equal to $+1$, while the remaining $n-k$ increments
are equal to $-1$. Consequently,
\begin{equation}\label{eq:conditional}
\begin{aligned}
\Pp(X_{n+1}=1\mid S_n=2k-n)
   &=\frac{1-a}{2}+\frac{ak}{n}.
\end{aligned}
\end{equation}
To reach $2k-(n+1)$ at time $n+1$, the walk must either be at
$2k-n$ at time $n$ and take a $-1$ step, or be at
$2(k-1)-n=2k-n-2$ at time $n$ and take a $+1$ step. Conditioning on these two possibilities and using \eqref{eq:conditional}, we obtain the known
recurrence
\begin{equation}\label{eq:massrec}
u_{n+1,k}
=
A_{n,k}u_{n,k}
+
B_{n,k}u_{n,k-1},
\end{equation}
where
\begin{equation}\label{eq:coefficients}
A_{n,k}
=
\frac{1+a}{2}-\frac{ak}{n},
\qquad
B_{n,k}
=
\frac{1-a}{2}+\frac{a(k-1)}{n}.
\end{equation}

\smallskip

\begin{proposition}\label{prop:lattice}
For each fixed $a\in(1/2,1)$ there are finite positive constants $C_0,C_1$ such that,
for every $q\in[0,1]$ and $n\geq1$,
\begin{equation}\label{eq:lattice-bounds}
 \max_k u_{n,k}\leq C_0n^{-a},
 \qquad
 \max_k|u_{n,k}-u_{n,k-1}|\leq C_1n^{-2a}.
\end{equation}
\end{proposition}

\begin{proof}
Set $M_n=\max_k u_{n,k}$. For $1\leq k\leq n$, both coefficients in
\eqref{eq:massrec} are nonnegative and by \eqref{eq:coefficients}
\[
 A_{n,k}+B_{n,k}=1-\frac an.
\]
Also recall that at $k=0$, $u_{n,-1}=0$ and $A_{n,0}=p=(1+a)/2$. At $k=n+1$,
$u_{n,n+1}=0$ and $B_{n,n+1}=p$. Since there exists $N_0 = N_0(a)$ such that $p \le 1- a/n$ for $n\ge N_0$, we also have 
\begin{equation}\label{eq:mass-contraction}
 M_{n+1}\leq\left(1-\frac an\right)M_n, \ \forall \, n\ge N_0.
\end{equation}

Put $d_{n,k}=u_{n,k}-u_{n,k-1}$. Subtracting the recurrence at $k-1$ from
the recurrence at $k$ gives
\begin{align*}
 d_{n+1,k}
 &=A_{n,k}u_{n,k}+(B_{n,k}-A_{n,k-1})u_{n,k-1}
   -B_{n,k-1}u_{n,k-2}.
\end{align*}
The identities
\[
 A_{n,k-1}=A_{n,k}+\frac an,
 \qquad B_{n,k}=B_{n,k-1}+\frac an,
\]
imply
\[
 B_{n,k}-A_{n,k-1}=B_{n,k-1}-A_{n,k}.
\]
Substitution and regrouping therefore yield the exact recurrence
\begin{equation}\label{eq:difference-rec}
 \begin{aligned}
 d_{n+1,k}
 &=A_{n,k}(u_{n,k}-u_{n,k-1})
   +B_{n,k-1}(u_{n,k-1}-u_{n,k-2})\\
 &=A_{n,k}d_{n,k}+B_{n,k-1}d_{n,k-1}.
 \end{aligned}
\end{equation}
The boundary cases of \eqref{eq:difference-rec} follow directly from the
definition of $u_{n,k}$. For
$1\leq k\leq n+1$,
\[
 A_{n,k}+B_{n,k-1}=1-\frac{2a}{n}.
\]
The smallest active coefficient in this range is
$(1-a)/2-a/n$, which, enlarging $N_0 = N_0(a)$ if necessary, is nonnegative for all $n \ge N_0$. At $k=0$ in \eqref{eq:difference-rec}, only
$A_{n,0}d_{n,0}=p\,d_{n,0}$ remains; at
$k=n+2$, only
$B_{n,n+1}d_{n,n+1}=p\,d_{n,n+1}$ remains. Setting $D_n=\max_k|d_{n,k}|$, we have
\begin{equation}\label{eq:difference-contraction}
 D_{n+1}\leq\left(1-\frac{2a}{n}\right)D_n , \ \forall \, n\ge N_0(a).
\end{equation}

If $y_{n+1}\leq(1-c/n)y_n$ for $n\geq N_0 >c$, then
\[
 y_n\leq y_{N_0}\prod_{j=N_0}^{n-1}\left(1-\frac cj\right).
\]
Since $\log(1-x)\leq-x$ for $0<x<1$,
\begin{align*}
 \log\prod_{j=N_0}^{n-1}\left(1-\frac cj\right)
 &\leq-c\sum_{j=N_0}^{n-1}\frac1j
 \leq-c\int_{N_0}^n\frac{dx}{x}
 =-c\log\frac nN_0.
\end{align*}
Exponentiating gives
\begin{equation}\label{eq:product-bound}
 \prod_{j=N_0}^{n-1}\left(1-\frac cj\right)\leq\left(\frac{N_0}{n}\right)^c.
\end{equation}

From \eqref{eq:product-bound}, $c=a$ in
\eqref{eq:mass-contraction} and
$c = 2a$ in \eqref{eq:difference-contraction}, 
\[
 M_n\leq M_{N_0} \frac{N_0^a}{n^a} \le \frac{N_0^a}{n^a}, \qquad
 D_n\leq D_{N_0} \frac{N_0^{2a}}{n^{2a}} \le \frac{N_0^{2a}}{n^{2a}},
 \qquad n\geq N_0.
\]
Thus the statement hold with $C_0 = N_0^a$ and $C_1 = N_0^{2a}$ which depend on $a$ but not on $q$.
\end{proof}

\section{The uniform superdiffusive local limit theorem}

The superdiffusive limit theorem for the classical ERW states
that, for $a\in(1/2,1)$ and $q\in[0,1]$,
\begin{equation}\label{eq:weak}
 \frac{S_n}{n^a}\xrightarrow{\mathrm{a.s.}} \mathcal{L}_{q,a},
\end{equation}
where $\mathcal{L}_{q,a}$ has a positive bounded smooth density $f_{q,a}$ on $\R$; see
\cite{baur-bertoin,bercu} and
\cite[Theorem 1.3]{guerin-laulin-raschel}. It is therefore natural to ask whether the convergence in \eqref{eq:weak} admits a local
counterpart. The following uniform local limit theorem, which is the main result of this section, provides an affirmative answer.

\smallskip

\begin{theorem}[Uniform superdiffusive local limit theorem]\label{thm:llt}
Let $a=2p-1\in(1/2,1)$ and $q\in[0,1]$. Then
\begin{equation}\label{eq:llt}
\lim_{n\to \infty} \sup_{\substack{j\in\Z\\ j\equiv n\ ({\rm mod}\ 2)}}
 \left|\frac{n^a}{2}\Pp(S_n=j)-f_{q,a}\!\left(\frac{j}{n^a}\right)\right|
 = 0.
\end{equation}
\end{theorem}

\smallskip

The first step in the proof is to regularize the rescaled lattice
distributions by triangular interpolation. Set $h_n:=2n^{-a}$, the
mesh size of the rescaled lattice, and
$x_{n,k}:=(2k-n)/n^a$ the lattice points. Let $T(y):=(1-|y|)_+$, and let $V$, independent
of $(S_n)_{n\geq1}$, have density $T$. Denote by $g_n$ the density of
$S_n/n^a+h_nV$, or equivalently of $(S_n+2V)/n^a$. Then
\begin{equation}\label{eq:interpolation}
 g_n(x)=\sum_{k=0}^n\frac{u_{n,k}}{h_n}
 T\!\left(\frac{x-x_{n,k}}{h_n}\right)
\end{equation}
Since $T$ is a probability density, $g_n$ is a continuous probability
density. Moreover, $h_nV\to0$ almost surely, so the probability
measures with densities $g_n$ have the same weak limit as
$S_n/n^a$, namely the law of $\mathcal{L}_{q,a}$.

At every lattice point,
\begin{equation}\label{eq:lattice-values}
 g_n(x_{n,k})=\frac{u_{n,k}}{h_n}=\frac{n^a}{2}u_{n,k}.
\end{equation}
Since $x_{n,k+1}-x_{n,k}=h_n$, the function $g_n$ is affine on
each interval $[x_{n,k},x_{n,k+1}]$, interpolating the two values in
\eqref{eq:lattice-values}. Its slope on this interval is
\[
 \frac{g_n(x_{n,k+1})-g_n(x_{n,k})}{h_n}
 =\frac{u_{n,k+1}-u_{n,k}}{h_n^2}.
\]
Since $u_{n,k}=0$ for
$k\notin\{0,\ldots,n\}$, the same expression applies to the two boundary intervals.

Proposition \ref{prop:lattice} and $h_n=2n^{-a}$ therefore give
$\|g_n\|_\infty\leq C_0/2$, while the absolute value of the slope on
every affine piece is bounded by $C_1/4$. Since $g_n$ is continuous
and piecewise affine, it is globally Lipschitz. Consequently,
\begin{equation}\label{eq:equi}
\sup_n\|g_n\|_\infty\leq\frac{C_0}{2} \quad \textrm{and}
\quad
|g_n(x)-g_n(y)|
\leq\frac{C_1}{4}|x-y|,
\qquad x,y\in\mathbb{R}.
\end{equation}

Now, to prove Theorem \ref{thm:llt} we will also need the next result, which is an immediate consequence of Boos
\cite[Lemma 1]{boos}.

\begin{lemma}\label{lem:upgrade}
Let $(\tilde g_n)$ be probability densities with a common uniform bound and a common
Lipschitz constant. If $\tilde g_n(x)\,dx$ converges weakly to a probability measure
with continuous density $\tilde f$, then $\tilde f$ is Lipschitz, vanishes at both infinities,
and
$\|\tilde g_n- \tilde f\|_\infty \xrightarrow[n\to \infty]{} 0.$
\end{lemma}

\begin{proof}[Proof of Theorem \ref{thm:llt}]
Recall that $g_n$ is the density of $Y_n=(S_n+2V)/n^a$. It follows from \eqref{eq:weak} that
$Y_n\xrightarrow{\mathrm{a.s.}}\mathcal{L}_{q,a}$.
Hence the probability measures with densities $g_n$ converge weakly
to the probability measure with density $f_{q,a}$. By \eqref{eq:equi}, the densities $(g_n)$ are uniformly bounded and
uniformly equicontinuous. Since $f_{q,a}$ is continuous, Lemma \ref{lem:upgrade} yields
$\|g_n-f_{q,a}\|_\infty \xrightarrow[n\to \infty]{} 0$.

Now let $j\in\mathbb{Z}$ satisfy $j\equiv n\pmod 2$ and set
$k=(j+n)/2\in\mathbb{Z}$. Under the convention that $u_{n,k}=0$ for
$k\notin\{0,\ldots,n\}$, the identity in
\eqref{eq:lattice-values} holds for every $k\in\mathbb{Z}$. Therefore,
\[
g_n\left(\frac{j}{n^a}\right)
=
\frac{u_{n,k}}{h_n}
=
\frac{n^a}{2}\Pp(S_n=j).
\]
Consequently,
\[
\sup_{\substack{j\in\mathbb{Z}\\ j\equiv n\pmod 2}}
\left|
\frac{n^a}{2}\Pp(S_n=j)
-
f_{q,a}\left(\frac{j}{n^a}\right)
\right|
\leq
\|g_n-f_{q,a}\|_\infty
\xrightarrow[n\to \infty]{} 0,
\]
which proves \eqref{eq:llt}. 
\end{proof}

\section{The breakdown of log-concavity of the limiting density}

As before $\mathcal{L}_a := \mathcal{L}_{1,a}$, whose density we denote as $f_a:=f_{1,a}$. Conditioning on the first increment and using reflection symmetry gives
the convex combination $f_{q,a}(x)=qf_{a}(x)+(1-q)f_{a}(-x)$. This is equivalent to the representation
$\mathcal{L}_{q,a}\stackrel{d}{=}(2\xi_q-1)\mathcal{L}_a$ with $\xi_q\sim\operatorname{Bernoulli}(q)$ independent of $\mathcal{L}_a$. Thus, the entire family of limiting densities is
determined by the case $q=1$. However, shape properties such as unimodality and log-concavity do not
automatically extend from $f_a$ to $f_{q,a}$, since they are not
preserved in general under convex mixtures. A natural first step is
therefore to study $f_a=f_{1,a}$ in detail. Accordingly, in what
follows, we restrict attention to $q=1$.

The main result of this section is that $f_a$ is not log-concave when $a$ is sufficiently close to $1$. Together with the transfer principle
established at Lemma \ref{lem:lc-transfer} just  below, this disproves the conjecture
$a_r\uparrow1$ of \cite{guerin-laulin-raschel-simon}. 

\subsection{The transfer principle for log-concavity.}

So we start with the transfer principle. Let us consider 
$\big(u_{n,k} : n\geq1, \, 1\leq k\leq n\big)$
as a triangular array of probabilities, with $n$ indexing the rows.
Thus, for each $n\geq1$, the $n$th row is
$(u_{n,k})_{1\leq k\leq n}$. We say that this row is log-concave if
\[
u_{n,k}^2\geq u_{n,k-1}u_{n,k+1},
\qquad 2\leq k\leq n-1.
\]
A passage from log-concavity of the finite-time distributions determined by the rows of the array to
log-concavity of the limiting density was considered in \cite[Corollary 2.4]{guerin-laulin-raschel-simon}. Our local limit
theorem yields the following subsequential transfer principle: it is
enough that log-concavity hold along an unbounded sequence of times.

\begin{lemma}\label{lem:lc-transfer}
Fix $a\in(1/2,1)$. If there are arbitrarily large times $n$ for which
$(u_{n,k})_{1\leq k\leq n}$ is log-concave, then $f_a$ is log-concave on $\R$.
\end{lemma}

\begin{proof}
The idea is to interpolate the logarithms of the rescaled p.m.f. Log-concavity of the rows makes these interpolations concave,
while Theorem \ref{thm:llt} makes them converge locally uniformly to $\log f_a$. Then the prove is complete, since locally uniform limit of concave functions is concave.

Choose a sequence $n_m\to\infty$ along which the corresponding rows
are log-concave, and set
\[
z_{m,k}:=\frac{n_m^a}{2}u_{n_m,k},
\qquad
r_{m,k}:=\log z_{m,k},
\qquad 1\leq k\leq n_m.
\]
These quantities are well defined because $q=1$ and $p\in(0,1)$ imply
that $u_{n_m,k}>0$ for every $1\leq k\leq n_m$. Log-concavity gives
\[
2r_{m,k}\geq r_{m,k-1}+r_{m,k+1},
\qquad 2\leq k\leq n_m-1.
\]
Since consecutive lattice points are separated by
$h_{n_m}=2n_m^{-a}$, the slopes between consecutive values of
$r_{m,k}$ are nonincreasing. Therefore, the piecewise affine
interpolation $r_m$ of the values $r_{m,k}$ at $x_{n_m,k}$ is concave
on
\[
I_m=[x_{n_m,1},x_{n_m,n_m}].
\]
Since $x_{n_m,1}\to-\infty$ and $x_{n_m,n_m}\to\infty$, every compact
interval is contained in $I_m$ for all sufficiently large $m$.

We now prove that $r_m\to\log f_a$ locally uniformly. Fix a compact
interval $K$ and set
$K^+:=\{x\in\mathbb{R}:\operatorname{dist}(x,K)\leq1\}$. Since $f_a$
is positive and continuous,
\[
\delta_K:=\min_{x\in K^+}f_a(x)>0.
\]
By Theorem \ref{thm:llt},
\[
\varepsilon_m(K)
:=
\max_{\substack{1\leq k\leq n_m\\x_{n_m,k}\in K^+}}
\left|z_{m,k}-f_a(x_{n_m,k})\right|
\xrightarrow[m\to \infty]{} 0.
\]
In particular, $z_{m,k}\geq\delta_K/2$ at all such lattice points for
all sufficiently large $m$. The mean-value theorem then gives
\[
\max_{\substack{1\leq k\leq n_m\\x_{n_m,k}\in K^+}}
\left|r_{m,k}-\log f_a(x_{n_m,k})\right|
\leq
\frac{2}{\delta_K}\varepsilon_m(K)
\xrightarrow[m\to \infty]{} 0.
\]

Let $\omega_K$ be the modulus of continuity of $\log f_a$ on $K^+$.
For sufficiently large $m$, the two lattice points used to interpolate
any $x\in K$ belong to $K^+$ and lie within distance $h_{n_m}$ of $x$.
It follows that
\[
\sup_{x\in K}|r_m(x)-\log f_a(x)|
\leq
\frac{2}{\delta_K}\varepsilon_m(K)
+\omega_K(h_{n_m})
\xrightarrow[m\to \infty]{} 0.
\]
Thus, $r_m\to\log f_a$ locally uniformly. For every compact interval
$K$, the function $r_m$ is concave on $K$ for all sufficiently large
$m$. Passing to the uniform limit shows that $\log f_a$ is concave on
$K$. Since $K$ is arbitrary, $f_a$ is log-concave on $\mathbb{R}$.
\end{proof}

\subsection{Failure of log-concavity for sufficiently strong memory}

Gu\'erin, Laulin and Raschel \cite[Theorem 1.4]{guerin-laulin-raschel}
proved the following moment recursion. Let $m_1=1$ and, for $k\geq2$, let
\begin{equation}\label{eq:moment-rec}
 m_k=\frac{1}{ka-c_k}\sum_{j=1}^{k-1}c_jm_jm_{k-j},
 \qquad
 c_j=\begin{cases}1,&j\text{ even},\\ a,&j\text{ odd}.
 \end{cases}
\end{equation}
Then
\begin{equation}\label{eq:moments-general}
 M_k^{(a)}:=\E[\mathcal{L}_a^k]=\frac{(k-1)!}{a\Gamma(ka)}m_k.
\end{equation}
Here $\Gamma(x)=\int_0^\infty t^{x-1}e^{-t}\,dt$. For $k=2$, since
$c_1=a$ and $c_2=1$, the recursion gives
\[
 m_2=\frac{c_1m_1^2}{2a-c_2}=\frac{a}{2a-1}.
\]
For $k=3$, $c_3=a$ and
\begin{align*}
 m_3
 &=\frac{c_1m_1m_2+c_2m_2m_1}{3a-c_3}
 =\frac{(a+1)m_2}{2a}
 =\frac{a+1}{2(2a-1)}.
\end{align*}
Substitution in \eqref{eq:moments-general}, together with
$a\Gamma(a)=\Gamma(1+a)$ for $k=1$, yields
\begin{equation}\label{eq:three-moments}
 M_1^{(a)}=\frac1{\Gamma(1+a)},\qquad
 M_2^{(a)}=\frac1{(2a-1)\Gamma(2a)},\qquad
 M_3^{(a)}=\frac{a+1}{a(2a-1)\Gamma(3a)}.
\end{equation}

Put $\varepsilon=1-a$. Write $\psi=\Gamma'/\Gamma$, then $\psi(x+1)=\psi(x)+\frac1x$. Since $\psi(1)=-\gamma$, where $\gamma$ is Euler-Mascheroni constant,
$\psi(2)=1-\gamma$ and $\psi(3)=3/2-\gamma$. Differentiating the logarithms of
the three exact moment formulas at $\varepsilon=0$ gives
\begin{align*}
 \left.\frac{d}{d\varepsilon}\log M_1^{(1-\varepsilon)}\right|_{\varepsilon=0}
   &=\psi(2)=1-\gamma,\\
 \left.\frac{d}{d\varepsilon}\log M_2^{(1-\varepsilon)}\right|_{\varepsilon=0}
   &=2+2\psi(2)=4-2\gamma,\\
 \left.\frac{d}{d\varepsilon}\log M_3^{(1-\varepsilon)}\right|_{\varepsilon=0}
   &=-\frac12+1+2+3\psi(3)=7-3\gamma,
\end{align*}
Therefore, with $b_1=1-\gamma$, $b_2=4-2\gamma$, and $b_3=7-3\gamma$,
\begin{align}
 M_j^{(a)}&=1+b_j\varepsilon+o(\varepsilon), \ j=1,2,3, \label{eq:mexp}
\end{align}
Consequently,
\[
 \operatorname{Var}(\mathcal{L}_a)
 =M_2^{(a)}-(M_1^{(a)})^2
 =(b_2-2b_1)\varepsilon+o(\varepsilon)
 =2\varepsilon+o(\varepsilon),
\]
and
\begin{align*}
 \E[(\mathcal{L}_a-\E \mathcal{L}_a)^3]
 &=M_3^{(a)}-3M_1^{(a)}M_2^{(a)}+2(M_1^{(a)})^3\\
 &=(b_3-3b_2+3b_1)\varepsilon+o(\varepsilon) = -2\varepsilon+o(\varepsilon).
\end{align*}
In particular,
\begin{equation}\label{eq:central-asymptotics}
 \operatorname{Var}(\mathcal{L}_a)=2\varepsilon+o(\varepsilon),
 \qquad
 \E[(\mathcal{L}_a-\E \mathcal{L}_a)^3]=-2\varepsilon+o(\varepsilon).
\end{equation}
Thus the skewness coefficient of $\mathcal{L}_a$ satisfies
\begin{equation}\label{eq:skew-diverges}
 \frac{\E[(\mathcal{L}_a-\E \mathcal{L}_a)^3]}{\operatorname{Var}(\mathcal{L}_a)^{3/2}}
 =-\frac{1+o(1)}{\sqrt{2\varepsilon}} = -\frac{1+o(1)}{\sqrt{2(1-a)}}\xrightarrow[a \to 1]{} -\infty.
\end{equation}

Applying Bubeck and Eldan \cite[Lemma 2]{bubeck-eldan} to $X$ and $-X$
gives the sharp skewness bound needed below: if $X$ is centered and has a log-concave density,
then
\begin{equation}\label{eq:eitan}
 \frac{|\E[X^3]|}{(\E[X^2])^{3/2}}\leq2.
\end{equation}
The constant is attained by $Z=E-1$, where $E\sim\operatorname{Exp}(1)$, and
by its reflection. Eitan \cite[Theorem 5]{eitan} later proved the following
sharp generalization to higher odd moments: if $1<r<s$ and $s$ is an odd
integer, then
\begin{equation}\label{eq:eitan-general}
 \frac{|\E[X^s]|^{1/s}}{(\E|X|^r)^{1/r}}
 \leq
 \frac{|\E[Z^s]|^{1/s}}{(\E|Z|^r)^{1/r}},
 \qquad Z=E-1,\quad E\sim\operatorname{Exp}(1).
\end{equation}
The choice $r=2$, $s=3$ recovers the Bubeck--Eldan bound.

\begin{proposition}[Failure of log-concavity]\label{thm:failure-near-one}
The parameter $a_\star$ defined in the Introduction is well-defined, satisfies
$a_{\rm ev}\leq a_\star<0.918$, and, for every $a\in(a_\star,1)$, the density
$f_a$ is not log-concave and every sufficiently late row is non-log-concave.
\end{proposition}

\begin{proof}
Let $X_a=\mathcal{L}_a-\E \mathcal{L}_a$. If $f_a$ were log-concave, then
\eqref{eq:eitan} would give
$|\E[X_a^3]|/\operatorname{Var}(\mathcal{L}_a)^{3/2} \leq2.$
Equation \eqref{eq:skew-diverges} shows that the set defining $a_\star$ is
nonempty. Put $V(a)=\operatorname{Var}(\mathcal{L}_a)$ and
$\mu_3(a)=\E[X_a^3]$. A $256$-bit outward-rounded Arb verification
\cite{johansson-arb}, using only \eqref{eq:three-moments}, certifies
$\mu_3(a)^2>4V(a)^3$ for every $a\in[0.918,1)$. It uses $2000$ rational
subintervals on $[0.918,0.99]$ and derivative bounds on $[0.99,1)$; the complete
certificate is available in the repository described in
Remark~\ref{rem:reproducibility}. Thus $f_a$ is not log-concave
throughout $[0.918,1)$. Strictness at $0.918$ and continuity give
$a_\star<0.918$. The defining set is upward closed, and Lemma
\ref{lem:lc-transfer} then yields eventual failure of row log-concavity for every
$a>a_\star$, hence $a_{\rm ev}\leq a_\star$.
\end{proof}

We next obtain a separate, much sharper bound for $a_{\rm ev}$ by a certified
analysis of the exact finite-time recurrence. Recall from the introduction that, by
\cite[Remark 2.3]{guerin-laulin-raschel-simon}, there exists an increasing sequence $(a_r)_{r\geq0}$ such that $a_{\rm ev}=\lim_{r\to\infty}a_r$ and
\begin{equation}\label{eq:threshold-structure}
 \text{all rows from time $r+3$ onward are log-concave}
 \quad\Longleftrightarrow\quad a\leq a_r.
\end{equation}

\begin{proposition}\label{prop:certified-upper}
At $a=0.80399$, every even row with $n\geq20000$ is non-log-concave.
Consequently, $a_{\rm ev}\leq 0.80399$.
\end{proposition}

\begin{proof} 
Fix $a=0.80399$ and $q=1$. If $\nabla^0 v_k = v_k$, $\nabla^1 v_k = \nabla v_k=v_k-v_{k-1}$ and $\nabla^\nu v_k= \nabla^{\nu-1} v_k- \nabla^{\nu-1} v_{k-1}$ for $\nu \ge 2$, repeated
subtraction in \eqref{eq:massrec}, using \eqref{eq:coefficients}, gives
\begin{equation}\label{eq:higher-difference}
 \nabla^\nu u_{n+1,k}=A_{n,k}\nabla^\nu u_{n,k}
 +B_{n,k-\nu}\nabla^\nu u_{n,k-1},\qquad 0\leq \nu \leq17.
\end{equation}
For $n+\nu$ even, put $c_\nu=(\nu+1)a$ and
\[
 W_n^{(\nu)}:=\frac{n^{c_\nu}}{2^{\nu+1}}\nabla^\nu u_{n,(n+\nu)/2}.
\]
Applying \eqref{eq:higher-difference} twice at the center gives
\begin{equation}\label{eq:centered-two-step}
W_{n+2}^{(\nu)}=Q_{n,\nu}W_n^{(\nu)}+R_{n,\nu}n^{-2a}W_n^{(\nu+2)},
\end{equation}
with
\[
 Q_{n,\nu}=\left(\frac{n+2}{n}\right)^{c_\nu}
 \left(1-\frac{c_\nu}{n+1}\right)\left(1-\frac{c_\nu}{n}\right),\qquad
 R_{n,\nu}=\left(\frac{n+2}{n}\right)^{c_\nu}
 \left(1-\frac{c_\nu}{n+1}\right)\left(1-\frac{c_\nu+a}{n}\right).
\]
Set $H_n=(W_{n+1}^{(1)})^2-W_n^{(0)}W_n^{(2)}$. A $256$-bit
outward-rounded Arb computation \cite{johansson-arb} of the exact recurrence to
$N=20000$, followed by interval propagation of \eqref{eq:centered-two-step}
through order $17$ and the bound
$\sum_{j\geq0}(N+2j)^{-2a}<0.001997$, certifies
\[
 H_N<-8.5125\times10^{-5},\qquad
 \sup_{\substack{n\geq N\\ n\text{ even}}}(H_n-H_N)<6.757\times10^{-5}.
\]
Thus $H_n<0$ for every even $n\geq N$. For such $n$, direct algebra from \eqref{eq:massrec} gives
\begin{equation}\label{eq:centered-determinant}
 \frac{n^{4a}}{16}\left(u_{n,n/2}^2-u_{n,n/2-1}u_{n,n/2+1}\right)
 =H_n+(W_{n+1}^{(1)})^2(P_n^{-2}-1)-n^{-2a}(W_n^{(2)})^2,
\end{equation}
where $P_n=(1+1/n)^{2a}(1-2a/n)$. The same certified enclosures give
\[
 (W_{n+1}^{(1)})^2(P_n^{-2}-1)\leq0.2495n^{-2},\qquad
 n^{-2a}(W_n^{(2)})^2\geq0.03782n^{-2a}.
\]
Since $a<1$, the latter dominates already at $n=N$ and hence for all
larger $n$. Therefore \eqref{eq:centered-determinant} is negative for
every even $n\geq N$. The complete outward-rounded verification is
included in the ancillary material, see Remark \ref{rem:reproducibility}.
Now fix $r\geq0$ and choose an even integer
$n\geq\max\{N,r+3\}$. Since the $n$th row is not log-concave, it is not
the case that all rows from time $r+3$ onward are log-concave.
Consequently, \eqref{eq:threshold-structure} gives $a>a_r$. Since $r$
was arbitrary,
so $a_{\rm ev}\leq a = 0.80399$.
\end{proof}


From the previous results we now have the main result of this section.

\begin{theorem}\label{thm:boundary}
The eventual-log-concavity boundary satisfies
\begin{equation}\label{eq:boundary}
\frac{\sqrt5-1}{2}
\leq a_{\rm ev}
\leq \min\{0.80399,a_\star\}
\leq a_\star<0.918.
\end{equation}
Moreover, the following dichotomy holds:
\begin{enumerate}
\item if $a<a_{\rm ev}$, all sufficiently late rows are log-concave
and $f_a$ is log-concave;
\item if $a>a_{\rm ev}$, non-log-concave rows occur at arbitrarily
large times.
\end{enumerate}
\end{theorem}

\begin{proof}
\cite[Proposition 2.2 and Remark 2.3]
{guerin-laulin-raschel-simon} give $a_0=(\sqrt5-1)/2$. Since $(a_r)_{r\geq0}$ is nondecreasing,
$a_0\leq a_{\rm ev}$, which gives the lower bound in
\eqref{eq:boundary}. Proposition \ref{prop:certified-upper} gives $a_{\rm ev}\leq0.80399$.

Suppose first that $a<a_{\rm ev}$. Then $a\leq a_r$ for some $r$, so
\eqref{eq:threshold-structure} implies that all rows from time $r+3$
onward are log-concave. Lemma \ref{lem:lc-transfer} then shows that $f_a$ is log-concave. It follows that $f_a$ is log-concave for every
$a<a_{\rm ev}$. Hence no $b<a_{\rm ev}$ belongs to the set defining $a_\star$. Therefore $a_{\rm ev}\leq a_\star$. Proposition \ref{thm:failure-near-one} gives
$a_\star<0.918$. Combining these inequalities yields \eqref{eq:boundary}.

Finally, if $a>a_{\rm ev}$, then $a>a_r$ for every $r$, and \eqref{eq:threshold-structure} gives non-log-concave rows at arbitrarily large times.
\end{proof}

\begin{remark}
The parameter $a_\star$ is an upper failure threshold, not a proved sharp
transition point for the limiting density. We have
$a<a_{\rm ev}\Rightarrow f_a$ log-concave and $a>a_\star\Rightarrow f_a$
non-log-concave, with no converse or assertion at either boundary.
\end{remark}

\section{Deterministic numerical evidence}

This section presents deterministic numerical evidence concerning $a_{\rm ev}$. For a fixed row, write
$\ell_{n,k}=\log u_{n,k}$ and define
\begin{equation}\label{eq:numerical-gap}
 \Delta_n(a)=\min_{2\leq k\leq n-1}
 \bigl(2\ell_{n,k}-\ell_{n,k-1}-\ell_{n,k+1}\bigr).
\end{equation}
The row is log-concave exactly when $\Delta_n(a)\geq0$. We evaluated the exact
recurrence \eqref{eq:massrec} in the log domain. More precisely, for every pair
of present terms put
\[
 \alpha_{n,k}=\log A_{n,k}+\ell_{n,k},\qquad
 \beta_{n,k}=\log B_{n,k}+\ell_{n,k-1},
\]
and compute
\begin{equation}\label{eq:logsumexp}
 \ell_{n+1,k}=m_{n,k}+
 \log\!\left(e^{\alpha_{n,k}-m_{n,k}}+e^{\beta_{n,k}-m_{n,k}}\right),
 \qquad
 m_{n,k}=\max\{\alpha_{n,k},\beta_{n,k}\}.
\end{equation}
For each fixed \(n\), we locate the sign change of \(a\mapsto\Delta_n(a)\)
by bisection, starting from \([0.500001,0.999999]\), and denote by
\(\beta_n\) the midpoint of the final bracket, whose width is at most
\(5\times10^{-8}\). Boundary terms are omitted, and a common additive
shift of the log-masses after each update prevents underflow without
changing \(\Delta_n\). Since \(n=r+3\), \(\beta_n\) is a numerical
approximation of \(a_r=a_{n-3}\). The computation is
deterministic, but does not provide a certified enclosure of \(a_{\rm ev}\). Table \ref{tab:thresholds} provides numerical approximations for some values of $r$.

\begin{table}[ht]
\centering
\caption{Finite-row sign-change values obtained by deterministic
bisection of the exact mass recurrence. Since \(r=n-3\), \(\beta_n\)
is a numerical approximation of \(a_r\). The first four values reproduce those
in \cite[Remark 2.3]{guerin-laulin-raschel-simon}.}
\label{tab:thresholds}
\begin{tabular}[t]{@{}cc@{}}
\hline
\(n\) & \(\beta_n\)\\
\hline
3   & 0.618033982\\
4   & 0.636065447\\
5   & 0.670605098\\
6   & 0.684081088\\
100 & 0.788605908\\
\hline
\end{tabular}
\hspace{1.5em}
\begin{tabular}[t]{@{}cc@{}}
\hline
\(n\) & \(\beta_n\)\\
\hline
500   & 0.798565924\\
1000  & 0.800452105\\
5000  & 0.802592201\\
10000 & 0.803018194\\
\hline
\end{tabular}
\end{table}
\begin{figure}[H]
\centering
\includegraphics[width=0.60\textwidth]{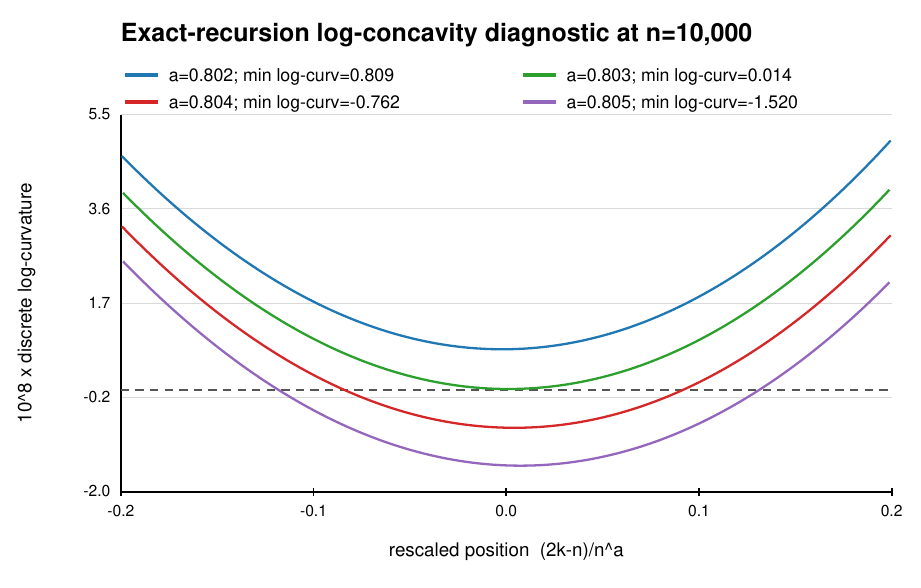}

\caption{Deterministic evaluation of $10^8$ times the discrete log-curvature in
\eqref{eq:numerical-gap} at $n=10000$. The dashed line is zero. A negative value
certifies failure of row log-concavity at that finite time. The observed sign
change near $a=0.804$ is numerical evidence.}
\label{fig:transition}
\end{figure}
Figure \ref{fig:transition} displays the discrete
log-curvature for four nearby parameters. At $n=10000$ its minimum is positive
for $a=0.802$ and $0.803$, but negative for $a=0.804$ and $0.805$. The slowly
increasing values in Table \ref{tab:thresholds}, together with the obtained
threshold, suggest
$a_{\rm ev}\approx0.80399$,
$p_{\rm ev}=(1+a_{\rm ev})/2\approx0.902$.




\begin{remark}[Reproducibility]\label{rem:reproducibility}
The code and reference outputs used for the certified Arb calculations in
Proposition~\ref{thm:failure-near-one} and Proposition~\ref{prop:certified-upper},
together with the deterministic computations for Table~\ref{tab:thresholds}
and Figure~\ref{fig:transition}, are available at
\url{https://github.com/heliotrinas-alt/superdiffusive-erw-local-limit}.
\end{remark}


\begin{thebibliography}{99}
\raggedright
\setlength{\itemsep}{0.35em}

\bibitem{schutz-trimper}
G.~M. Sch\"utz and S. Trimper,
\newblock Elephants can always remember: exact long-range memory effects in a
non-Markovian random walk,
\newblock \emph{Phys. Rev. E} \textbf{70} (2004), 045101.

\bibitem{baur-bertoin}
E. Baur and J. Bertoin,
\newblock Elephant random walks and their connection to P\'olya-type urns,
\newblock \emph{Phys. Rev. E} \textbf{94} (2016), 052134.

\bibitem{bercu}
B. Bercu,
\newblock A martingale approach for the elephant random walk,
\newblock \emph{J. Phys. A: Math. Theor.} \textbf{51} (2018), 015201.

\bibitem{kubota-takei}
N. Kubota and M. Takei,
\newblock Gaussian fluctuation for superdiffusive elephant random walks,
\newblock \emph{J. Stat. Phys.} \textbf{177} (2019), 1157--1171.

\bibitem{fan-hu-ma}
X. Fan, H. Hu and X. Ma,
\newblock Cram\'er moderate deviations for the elephant random walk,
\newblock \emph{J. Stat. Mech. Theory Exp.} (2021), 023402.

\bibitem{guerin-laulin-raschel}
H. Gu\'erin, L. Laulin and K. Raschel,
\newblock A fixed-point equation approach for the superdiffusive elephant random
walk,
\newblock \emph{Ann. Inst. H. Poincar\'e Probab. Statist.} \textbf{62} (2026),
973--1005.

\bibitem{guerin-laulin-raschel-simon}
H. Gu\'erin, L. Laulin, K. Raschel and T. Simon,
\newblock On the limit law of the superdiffusive elephant random walk,
\newblock \emph{Electron. J. Probab.} \textbf{30} (2025), Paper No. 102, 1--25.

\bibitem{boos}
D.~D. Boos,
\newblock A converse to Scheff\'e's theorem,
\newblock \emph{Ann. Statist.} \textbf{13} (1985), 423--427.

\bibitem{bubeck-eldan}
S. Bubeck and R. Eldan,
\newblock The entropic barrier: a simple and optimal universal self-concordant
barrier,
\newblock arXiv:1412.1587 (2014), Lemma 2.

\bibitem{eitan}
Y. Eitan,
\newblock The centered convex body whose marginals have the heaviest tails,
\newblock \emph{Studia Math.} \textbf{274} (2024), 201--215.

\bibitem{johansson-arb}
F. Johansson,
\newblock Arb: efficient arbitrary-precision midpoint-radius interval arithmetic,
\newblock \emph{IEEE Trans. Comput.} \textbf{66} (2017), 1281--1292.

\bibitem{peres-qin}
Y. Peres and S. Qin,
\newblock Transition probabilities of step-reinforced random walks,
\newblock arXiv:2604.07227v1 (2026).

\end{thebibliography}
\end{document}